\documentclass[a4paper,11pt]{article}

\usepackage{amsmath,amssymb,amsthm}
\usepackage{amssymb,latexsym, bbm,comment,url}

\renewcommand{\theequation}                            
       {\mbox{\arabic{section}.\arabic{equation}}}

{\theoremstyle{plain}
\newtheorem{definition}{Definition}[section]
\newtheorem{lemma}[definition]{Lemma}
\newtheorem{theorem}[definition]{Theorem}
\newtheorem{corollary}[definition]{Corollary}

}
{\theoremstyle{definition}

\newtheorem{remark}[definition]{Remark}

}

\renewcommand{\mathbb}{\mathbbm}                     
\renewcommand{\epsilon}{\varepsilon}                 
\renewcommand{\phi}{\varphi}
\renewcommand{\theta}{\vartheta}
\renewcommand{\le}{\leqslant}
\renewcommand{\ge}{\geqslant}

\newcommand{\origfoo}{} \let\origfoo=\sqrt           
\renewcommand{\sqrt}[1]{\origfoo{#1}\;}

\newcommand{\abs}[1]{\left\lvert #1 \right\rvert}    
\newcommand{\norm}[1]{\left\lVert #1 \right\rVert}   

\DeclareMathOperator{\F}{{\cal F}}                   
\DeclareMathOperator{\R}{{\mathbb R}}                
\DeclareMathOperator{\C}{{\mathbb C}}                
\DeclareMathOperator{\N}{{\mathbb N}}                
\DeclareMathOperator{\I}{I}                          %

\DeclareMathOperator{\Borel}{{\mathcal B}}
\newcommand{\scapro}[2]{\langle #1,#2\rangle}       
\DeclareMathOperator{\1}{\mathbbm 1}
\renewcommand{\S}{{\mathcal S}}

\renewcommand{\L}{{\mathcal L}}
\renewcommand{\H}{{\mathcal H}}
\DeclareMathOperator{\Z}{{\mathcal Z}}

\DeclareMathOperator{\Cc}{{\mathcal C}}

\title{Stochastic integration
with respect to \\cylindrical L{\'e}vy processes in Hilbert spaces: \\an $L^2$ approach}

\author{ Markus Riedle\footnote{markus.riedle@kcl.ac.uk}{}\\
Department of Mathematics\\
King's College\\
London WC2R 2LS\\
United Kingdom }
\begin{document}
\maketitle
\begin{abstract}
In this work stochastic integration with respect to cylindrical L{\'e}vy processes with weak second moments
is introduced. It is well known that a deterministic Hilbert-Schmidt operator
radonifies a cylindrical random variable, i.e. it maps a cylindrical random variable to a classical Hilbert space valued
random variable. Our approach is based on a generalisation of this result to
the radonification of the cylindrical increments of a cylindrical L{\'e}vy process  by {\em random} Hilbert-Schmidt operators.
This generalisation enables us to introduce a Hilbert space valued random variable as the
stochastic integral of a predictable stochastic process with respect to a cylindrical
L{\'e}vy process.
We finish this work by deriving an It{\^o} isometry  and by considering shortly
stochastic partial differential equations driven by cylindrical L{\'e}vy processes.
\end{abstract}
\hspace*{0.5cm}


\section{Introduction}

Stochastic differential equations in Hilbert spaces and as their very prerequisite,  stochastic integrals in Hilbert spaces,  have been extensively studied since the 1960s. One can find various early publications for different classes of integrands and different integrators, see for example Curtain and Falb \cite{CurtainFalb} for integration with respect to  Wiener processes, Kunita \cite{Kunita},  M{\'e}tivier and Pistone \cite{MetivierPistone} for integration with respect to square integrable martingales and M{\'e}tivier \cite{Metivier82}  for integration with respect to semi-martingales.

Since there does not exist a Wiener process with independent components along the basis of an infinite dimensional Hilbert space, the most common models of random noises are either the so-called cylindrical Wiener process or the space-time white noise (which can be seen as a special case of the cylindrical Wiener process). Integration with respect to  cylindrical Wiener processes is developed for example in
Daletskij \cite{Daletskij}, followed by the articles Gaveau \cite{Gaveau}, Lepingle and Ouvrard \cite{LepingleOuvrard} and others.
Integration with respect to space-time white noise is derived in the seminal work by Walsh \cite{Walsh}. The recent work by Dalang and Quer-Sardanyons \cite{DalangQuer} compares both approaches.

Surprisingly, integration with respect to other cylindrical processes than the cylindrical Wiener process is only considered in a few works. In fact, we are only aware of two approaches to integration with respect to cylindrical martingales, which origin either in the work developed by  M{\'e}tivier  and Pellaumail in  \cite{MetivierPellcylindrical} and \cite{MetivierPell} or
by Mikulevi\v{c}ius and Rozovski\v{\i} in \cite{MikRoz98} and \cite{MikRoz99}.
The construction by  M{\'e}tivier  and Pellaumail is based on  Dol\'eans measures
whereas Mikulevi\v{c}ius and Rozovski\v{\i} use a family of reproducing kernel Hilbert spaces. Thus, both constructions rely heavily on the assumed existence of weak second moments. In M{\'e}tivier  and Pellaumail \cite{MetivierPellcylindrical}, the construction is extended to cylindrical local martingales.

The increasing attention to models with L{\'e}vy noise led recently to
the concept of {\em cylindrical L{\'e}vy processes}. The notion {\em cylindrical L{\'e}vy process} appears the first time in Peszat and Zabczyk \cite{PeszatZab} and it is followed by the works of Brze\'zniak et al \cite{Zdzisetal}, Brze\'zniak and Zabzcyk \cite{ZdzisJerzy} and Priola and Zabczyk \cite{PriolaZabczyk}. The first systematic
introduction of cylindrical L{\'e}vy processes is presented in our work
Applebaum and Riedle \cite{DaveMarkus}. In this article \cite{DaveMarkus} we introduce a very general, discontinuous noise occurring in  time and state space as a cylindrical L{\'e}vy process,  which is a natural generalisation of a cylindrical Wiener process. We follow in \cite{DaveMarkus} the classical approach to cylindrical measures and cylindrical processes, as presented for example in Badrikian \cite{Badrikian} and Schwartz \cite{Schwartz}.

The main objective of this work is to establish a theory of stochastic integration with respect to  cylindrical L{\'e}vy processes with weak second moments.
Although this situation is covered by the above-mentioned more general work by M{\'e}tivier and Pellaumail (\cite{MetivierPellcylindrical}, \cite{MetivierPell}) and Mikulevi\v{c}ius and Rozovski\v{\i} (\cite{MikRoz98}, \cite{MikRoz99}),
the special case of cylindrical L{\'e}vy processes deserves their own treatment.
For,  the independent and stationary increments of a cylindrical L{\'e}vy process enable us to develop a straightforward integration theory in this work, which does not only mimic the approach in case of a classical L{\'e}vy process but it is also based on much easier and more familiar arguments.
Our approach immediately exposes the need to restrict the class of admissible integrands to the set of stochastic processes with values in the space of Hilbert-Schmidt operators. However more important, in contrary to the above-mentioned more general work, our construction of the stochastic integral does not rely on the weak second moments but only the proofs are based on this property.
For this reason, we expect that our construction in this work will serve as the basis to develop an integration theory with respect to cylindrical L{\'e}vy processes without assuming finite weak second moments.
Since a cylindrical L{\'e}vy process can not be decomposed into a cylindrical local martingale and
and a cylindrical bounded variation process, this situation is not covered by the  work mentioned above.

The construction of the stochastic integral is based on the following observation:
 the increments of a cylindrical process are probabilistically described by cylindrical measures, i.e. finitely additive set functions whose projections to Euclidean spaces are always probability measures. It is well known, as a consequence of
 Sazonov's theorem, that the image cylindrical measure $\mu\circ T^{-1}$ of a cylindrical measure $\mu$ under a linear operator $T$ extends to a Radon measure if $T$ is a Hilbert-Schmidt operator. Thus, by following It{\^o}'s approach to stochastic integration by firstly integrating simple stochastic processes,
  the extension of this approach to stochastic integration with respect to cylindrical processes requires a generalisation of this result to {\em random} Hilbert-Schmidt operators. This approach is motivated by the work of Jakubowski et al \cite{Jaketal}, where the authors show that a single deterministic Hilbert-Schmidt operator transfers a cylindrical semi-martingale to a classical semi-martingale in the underlying Hilbert space.

In Section  \ref{se.RandomHS} we show that indeed a random Hilbert-Schmidt operator maps the cylindrical increments to a classical Radon random variable in the Hilbert space. Once this result is derived we follow the classical approach and define an integral operator on the space of simple stochastic processes. By showing
the continuity of this operator in Section \ref{se.integration},
we can define the stochastic integral on the completion of the space of simple processes under an appropriate  topology. To illustrate the usage and applicability of the developed theory, we  consider stochastic partial differential equations driven by cylindrical L{\'e}vy processes in Section \ref{se.diff}.

With referring to the celebrated Bichteler-Dellacherie-Mokobodzki Theorem, our results can be summarised as identifying
 the cylindrical L{\'e}vy process as a {\em good integrator}. The stochastic integration of simple stochastic processes with respect
 to the cylindrical L{\'e}vy process
defines a continuous and linear operator. This property is assumed
to be satisfied in the work by Kurtz and Protter \cite{KurtzProtter96}, in order to derive weak limit results of
 stochastic integrals with respect to infinite dimensional semi-martingales and their generalisations, the so-called $H^{\#}$-semi-martingales, which also include cylindrical semi-martingales.

\section{Preliminaries}

Let $U$ be a separable Hilbert space with scalar product $\scapro{\cdot}{\cdot}$,
norm $\norm{\cdot}_U$ and orthonormal basis $\{e_k\}_{k\in\N}$. The Borel $\sigma$-algebra in $U$ is denoted by  $\Borel(U)$ and the closed unit ball at the origin by $B_U:=\{u\in U:\, \norm{u}_U\le 1\}$. The dual space is identified with $U$.

For a measure space $(S,\S,\mu)$ we denote by $L^p_\mu(S,\S;U)$, $p\ge 0$, the space of
equivalence classes of measurable functions $f:S\to U$ with $\int
\norm{f(s)}_U^p\,\mu(ds)<\infty$.

For every $u_1,\dots, u_n\in U$ and $n\in\N$ we define a linear map
\begin{align*}
  \pi_{u_1,\dots, u_n}:U\to \R^n,\qquad
   \pi_{u_1,\dots, u_n}(u)=(\scapro{u}{_1},\dots,\scapro{u}{u_n}).
\end{align*}
Let $\Gamma$ be a subset of $U$. Sets of the form
 \begin{align*}
Z(u_1,\dots, u_n;B):&= \{u\in U:\, (\scapro{u}{u_1},\dots,
 \scapro{u}{u_n})\in B\}\\
 &= \pi^{-1}_{u_1,\dots, u_n}(B),
\end{align*}
where $u_1,\dots, u_n\in \Gamma$ and $B\in \Borel(\R^n)$ are called {\em cylindrical sets}. The set of all cylindrical sets is denoted by $\Z(U,\Gamma)$ and it is an algebra. The generated $\sigma$-algebra is denoted by
$\Cc(U,\Gamma)$ and it is called the {\em cylindrical $\sigma$-algebra with
respect to $(U,\Gamma)$}. If $\Gamma=U$ we write $\Z(U):=\Z(U,\Gamma)$ and
$\Cc(U):=\Cc(U,\Gamma)$.

A function $\mu:\Z(U)\to [0,\infty]$ is called a {\em cylindrical measure on
$\Z(U)$}, if for each finite subset $\Gamma\subseteq U$ the restriction of
$\mu$ to the $\sigma$-algebra $\Cc(U,\Gamma)$ is a measure. A cylindrical
measure is called finite if $\mu(U)<\infty$ and a cylindrical probability
measure if $\mu(U)=1$.
For every function $f:U\to\C$ which is measurable with respect to
$\Cc(U,\Gamma)$ for a finite subset $\Gamma\subseteq U$ the integral $\int
f(a)\,\mu(da)$ is well defined as a complex valued Lebesgue integral if it
exists. In particular, the characteristic function $\phi_\mu:U\to\C$ of a
finite cylindrical measure $\mu$ is defined by
\begin{align*}
 \phi_{\mu}(u):=\int_U e^{i\scapro{u}{a}}\,\mu(du)\qquad\text{for all }a\in
 U^\ast.
\end{align*}

Let $(\Omega, \F,P)$ be a complete probability space.
Similarly to the correspondence between measures and random variables, there is
an analogous random object associated to cylindrical measures:
a {\em cylindrical random variable $Y$ in $U$} is a linear and continuous map
\begin{align*}
 Y:U \to L^0_P(\Omega,\F;\R),
\end{align*}
where the space $L^0_P(\Omega,\F;\R)$ is equipped with the topology of convergence in probability.
A cylindrical process $Y$ in $U$ is a family $(Y(t):\,t\ge 0)$ of
cylindrical random variables in $U$.

The characteristic function of a cylindrical random  variable $Y$ is
defined by
\begin{align*}
 \phi_Y:U\to\C, \qquad \phi_Y(u)=E[\exp(iYu)].
\end{align*}
If $Z=Z(u_1,\dots, u_n;B)$ is a cylindrical set for
$u_1,\dots, u_n\in U$ and $B\in \Borel(\R^n)$ we obtain a cylindrical
probability measure $\mu$ by the prescription
\begin{align*}
  \mu(Z):=P((Yu_1,\dots, Yu_n)\in B).
\end{align*}
We call $\mu$ the {\em cylindrical distribution of $Y$} and the
characteristic functions $\phi_\mu$ and $\phi_Y$ of $\mu$ and $Y$
coincide. Conversely, for every cylindrical measure $\mu$ on
$\Z(U)$ there exists a probability space $(\Omega,\F,P)$ and a
cylindrical random variable $Y:U\to L^0_P(\Omega,\F;\R)$ such
that  $\mu$ is the cylindrical distribution of $Y$, see
\cite[VI.3.2]{Vaketal}.

In the work \cite{DaveMarkus} we introduce the concept of {\em cylindrical L{\'e}vy processes}. A
cylindrical process $(L(t):\, t\ge 0)$ in $U$ is called a cylindrical L{\'e}vy process
if for every $u_1,\dots, u_n\in U$ and $n\in\N$ the stochastic process
\begin{align*}
  \big((L(t)u_1,\dots, L(t)u_n):\, t\ge 0\big)
\end{align*}
is a L{\'e}vy process in $\R^n$. In the following we will equip the probability space $(\Omega,\F,P)$
with the filtration $\{\F_t\}_{t\ge 0}$ generated by a cylindrical L{\'e}vy process and
defined by
\begin{align*}
  \F_t:=\sigma(\{L(s)u:\, u\in U, \, s\in [0,t]\})
  \qquad\text{for all }t\ge 0.
\end{align*}

The definition of a cylindrical L{\'e}vy process is a straightforward generalisation of a cylindrical Wiener process
if the latter is defined analogously: a cylindrical process
$(W(t):\, t\ge 0)$ is called a {\em cylindrical Wiener process} if
for every $u_1,\dots, u_n$ and $n\in\N$ the stochastic process
\begin{align*}
  \big((W(t)u_1,\dots, W(t)u_n):\, t\ge 0\big)
\end{align*}
is a Wiener process in $\R^n$. Here we call an adapted stochastic process $(X(t) : t \ge 0)$
in $\R^n$ a Wiener process if the increments $X(t) - X(s)$ are independent, stationary and
normally distributed with expectation $E[X(t) - X(s)] = 0$ and covariance $Cov[X(t) -
X(s),X(t)-X(s)] = |t - s|C$ for a non-negative, definite symmetric matrix $C$. This approach of defining
a cylindrical Wiener process can be found for example in the monographs  \cite{KallianpurXiong} and
\cite{MetivierPell} and it is recently reviewed in \cite{Riedle:wiener}.

The characteristic function $\phi_{L(t)}:U\to\C$, $t\ge 0$ of a cylindrical L{\'e}vy process $L$ is given by
\begin{align}\label{eq.charth}
 \phi_{L(t)}(u)&\\
 = \exp &\left(t\left(ip(u)- \tfrac{1}{2}q(u)+\int_U \left(e^{i\scapro{u}{a}}-1-i\scapro{u}{a}
   \1_{B_{\R}}(\scapro{u}{a})\right)   \, \nu (da)\right)\right),\notag
\end{align}
where $p:U\to\R$ is a continuous mapping, $q:U\to \R$ is a quadratic form and $\nu$ is a cylindrical measure on $\Z(U)$ such that $\nu\circ \scapro{\cdot}{u}^{-1}$ is the L{\'e}vy measure on $\Borel(\R)$ of $(L(t)u:\, t\ge 0)$ for each
$u\in U$. This result is derived in \cite{DaveMarkus} with some
addenda in \cite{Riedle}. The triplet $(p,q,\nu)$ is called the
{\em cylindrical characteristics of $L$}.

In this work we only consider cylindrical L{\'e}vy processes $(L(t):\, t\ge 0)$  with existing weak second moments that
is $E[\abs{L(1)u}^2 ]<\infty$ for all $u\in U$. In this case, the closed mapping theorem implies that
the operator $L(t):U\to L^2_P(\Omega,\F;\R)$ is continuous for all $t\ge 0$. We call the cylindrical process $(L(t):\, t\ge 0)$ a {\em cylindrical martingale} if $(L(t)u:\, t\ge 0)$ is a real-valued martingale for every $u\in U$.

\section{Random Hilbert-Schmidt operators}\label{se.RandomHS}

Let $V$ be another separable Hilbert space with scalar product $\scapro{\cdot}{\cdot}$, norm $\norm{\cdot}_V$ and  orthonormal basis $\{f_k\}_{k\in\N}$. We also identify the dual space of $V$ with $V$.

 The space of all linear bounded operators $\phi:U\to V$ is denoted by $\L(U,V)$. An operator
$\phi\in \L(U,V)$ is called a Hilbert-Schmidt operator if
\begin{align*}
  \sum_{k=1}^\infty \norm{\phi e_k}_V^2< \infty.
\end{align*}
The space of all Hilbert-Schmidt operators is denoted by $\L_2(U,V)$ and it can be equipped
with a scalar product defined by
\begin{align*}
  \scapro{\phi}{\psi}:=\sum_{k=1}^\infty \scapro{\phi e_k}{\psi e_k}.
\end{align*}
It follows that $\L_2(U,V)$ is complete and thus a Hilbert space with norm
\begin{align*}
  \norm{\phi}_{\L_2}:=\left(\sum_{k=1}^\infty \norm{ \phi e_k}_V^2\right)^{1/2}.
\end{align*}

Let $(L(t):\, t\ge 0)$ be a cylindrical L{\'e}vy process in $U$ with weak second moments. We consider
for fixed times $0\le s\le t$ the cylindrical random variable $\Delta L:=L(t)-L(s):U\to L^2_P(\Omega,\F;\R)$.
The space of $2$-radonifying operators coincides with the space of Hilbert-Schmidt operators in Hilbert spaces. Therefore, for every $\phi\in \L_2(U,V)$ there exists a random variable $X_\phi\in L^2_P(\Omega,\F;V)$ such that
\begin{align}\label{eq.radoni}
  (\Delta L)(\phi^\ast v)= \scapro{X_\phi}{v} \qquad\text{for all }v\in V.
\end{align}
By linearity of $\Delta L$,  the equality \eqref{eq.radoni} can be generalised to random Hilbert-Schmidt operators if they are simple. An $\L_2(U,V)$-valued, $\F_s$-measurable random variable $\Phi$
is called {\em simple} if it is of the form
\begin{align}\label{eq.simple}
  \Phi(\omega)=\sum_{i=1}^m \1_{A_i}(\omega) \phi_i\qquad\text{for all }\omega\in\Omega,
\end{align}
for disjoint sets $A_1,\dots, A_m\in \F_s$, $\phi_1,\dots, \phi_m\in \L_2(U,V)$ and $m\in\N$. The space of all $\L_2(U,V)$-valued, $\F_s$-measurable, simple random variables is
denoted by $S(\Omega, \F_s;\L_2)$. For later reference we equip the space $S(\Omega, \F_s;\L_2)$ with the semi-norm
\begin{align*}
  \norm{\Phi}_S:=\left(E\left[\norm{\Phi}_{\L_2}^2\right]\right)^{1/2}.
\end{align*}

If $\Phi$ is of the form \eqref{eq.simple} then \eqref{eq.radoni} implies  that for each $i=1,\dots, m$
 there exists a random
variable $X_{\phi_i}\in L^2_P(\Omega,\F; V)$ such that  $(\Delta L)(\phi_i^\ast v)=\scapro{X_{\phi_i}}{v}$ for all $v\in V$ and we can define a random variable
\begin{align*}
  J_{s,t}(\Phi):\Omega\to V,
  \qquad J_{s,t}(\Phi)=\1_{A_1}X_{\phi_1}+\dots +\1_{A_m}X_{\phi_m}.
\end{align*}
In this way, one obtains a linear operator
\begin{align}\label{eq.defJ}
  J_{s,t}: S(\Omega, \F_s;\L_2)\to L^2_P(\Omega,\F;V).
\end{align}
For each $v\in V$ the random variable $J_{s,t}(\Phi)$ satisfies
\begin{align*}
\scapro{J_{s,t}(\Phi)}{v}
 =\sum_{i=1}^m \1_{A_i} \scapro{X_{\phi_i}}{v}
 =\sum_{i=1}^m \1_{A_i} (\Delta L)(\phi_i^\ast v),
\end{align*}
which, together with the linearity of $\Delta L$, motivates us to define
\begin{align}\label{eq.defincL}
  \big(L(t)-L(s)\big)(\Phi^\ast v):= \scapro{J_{s,t}(\Phi)}{v}
  \qquad\text{for all }v\in V.
\end{align}
\begin{remark}
  Our definition in \eqref{eq.defincL} of the application of the cylindrical random variable $\Delta L$ to the
  $U$-valued random variable $\Phi^\ast v$ is based on the linearity of $\Delta L$ and the simple form
  of $\Phi^\ast v$.  However, in general it is not possible to define the application of a cylindrical random variable $Z:U\to L^0_P(\Omega,\F;\R)$ to an $U$-valued random variable $Y:\Omega\to U$  $\omega$-wise.
For, if $\omega\in\Omega$ is fixed the random variable $Z(Y(\omega))$ can be arbitrarily defined on
a null set $N_\omega$ depending on $\omega$. Consequently, the mapping $\omega\to Z(Y(\omega))(\omega)$
might be arbitrarily defined on a set which is not a null set if $\Omega$ is not countable.
\end{remark}

\begin{theorem}\label{th.radonifying}
  For $0\le s\le t$  let $J_{s,t}$ be the operator defined in \eqref{eq.defJ}. Then $J_{s,t}$ is
  continuous and satisfies
\begin{align*}
    \norm{J_{s,t}}_{S\to L^2_P}\le \norm{L(t-s)}_{U\to L^2_P}.
\end{align*}
\end{theorem}
\begin{proof}
  Let $\Phi\in S(\Omega, \F_s;\L_2)$ be of the form \eqref{eq.simple}. The independence of $\Delta L:=L(t)-L(s)$ and $\F_s$
  implies for each $v\in V$ that
  \begin{align}\label{eq.condex}
    E\left[\abs{(\Delta L)(\Phi^\ast v)}^2\right]
    &= E\left[ \abs{\sum_{i=1}^m \1_{A_i} (\Delta L)(\phi_i^\ast v)}^2\right]\notag\\
    &= \sum_{i=1}^m  E\left[ \abs{\1_{A_i} (\Delta L)(\phi_i^\ast v)}^2\right]\notag\\
     &= \sum_{i=1}^m E\left[ E\left[\abs{\1_{A_i} (\Delta L)(\phi_i^\ast v)}^2|\F_s\right]\right]\notag\\
    &= \sum_{i=1}^m \int_{\Omega} E\left[ \abs{\1_{A_i}(\omega) (\Delta L)(\phi_i^\ast v)}^2\right] \, P(d\omega)\notag\\
   &= \int_{\Omega} E\left[ \abs{(\Delta L)(\Phi^\ast(\omega) v)}^2\right] \, P(d\omega).
  \end{align}
Applying this equality results in
\begin{align*}
  E\left[\norm{J_{s,t}(\Phi)}_V^2\right]
    &= \sum_{j=1}^\infty E\left[\abs{\scapro{J_{s,t}(\Phi)}{f_j}}^2\right]\\
  &= \sum_{j=1}^\infty E\left[\abs{(\Delta L)(\Phi^\ast f_j)}^2\right]\\
   &= \sum_{j=1}^\infty \int_{\Omega} E\left[ \abs{(\Delta L)(\Phi^\ast(\omega) f_j)}^2\right] \, P(d\omega)\\
  &\le \norm{\Delta L}_{U\to L^2_P}^2\sum_{j=1}^\infty E\left[\norm{\Phi^\ast f_j}_U^2\right]\notag\\
  &= \norm{\Delta L}_{U\to L^2_P}^2 E\left[\norm{\Phi}_{\L_2}^2\right],
\end{align*}
which completes the proof.
\end{proof}

We denote the Bochner space $L^2_P(\Omega,\F_s;\L_2(U,V))$ by $L^2_P(\Omega,\F_s;\L_2)$.
Since the operator $J_{s,t}$ is continuous on $S(\Omega,\F_s;\L_2)$  and due to the well known fact that $S(\Omega,\F_s;\L_2)$
is dense in the Bochner space $L^2_P(\Omega,\F_s;\L_2)$ we can uniquely extend the operator $J_{s,t}$ to the
latter space. In this way, for each $\Phi\in L^2_P(\Omega,\F_s;\L_2)$ we define
\begin{align*}
  (L(t)-L(s))(\Phi^\ast v):=
  \scapro{J_{s,t}\Phi}{v}
   \qquad\text{for all }v\in V.
\end{align*}

\begin{remark}
  A Hilbert-Schmidt operator $\phi\in \L_2(U,V)$ is not only 2-radonifying but even 0-radonifying. That means, even without  requiring that $L$ has weak second moments,   there exists
   a classical random variable $X_\phi:\Omega\to V$ such that
 $ (L(t)-L(s))(\phi^\ast v)= \scapro{X_\phi}{v}$ for all $v\in V$. However, the $V$-valued random variable
 $X_\phi$ might not have finite moments.
One can expect that the analogue result of Theorem \ref{th.radonifying} can be modified such that the $P$-a.s. convergence of a sequence $\{\Phi_n\}_{n\in\N}\subseteq S(\Omega,\F_s;\L_2)$ to $\Phi\in L^0_P(\Omega,\F_s;\L_2)$ implies that $\{J_{s,t}(\Phi_n)\}_{n\in\N}$ is a Cauchy sequence in $L_P^0(\Omega,\F;V)$, equipped with the topology of convergence in probability.
\end{remark}

\section{Stochastic integration}\label{se.integration}

An $\L_2(U,V)$-valued, stochastic process $(\Psi(t):\, t\in [0,T])$ is called {\em simple} if it is of the form
\begin{align}\label{eq.defsimple}
  \Psi(t)=\sum_{k=0}^{N-1} \Phi_k \1_{(t_k,t_{k+1}]}(t),
\end{align}
where $0=t_0< t_1< \cdots < t_N=T$ is a finite sequence of
deterministic times  and $\Phi_k$ is a random variable  in $L^2_P(\Omega,\F_{t_k};\L_2)$ for each
$k=0,\dots, N-1$.
The set of all simple $\L_2(U,V)$-valued stochastic processes is denoted by $\H_0(\L_2)$ and it is
endowed with the semi-norm $\norm{\cdot}_{\H}$ defined by
\begin{align*}
  \norm{\Psi}_{\H}:=\left(E\left[\int_0^T \norm{\Psi(s)}_{\L_2}^2\, ds\right]\right)^{1/2}.
\end{align*}
It is well known (see for example \cite[Pro. 4.7]{DaPratoZab}) that $\H_0(\L_2)$ is dense
in the space
\begin{align*}
\H(\L_2):=
\left\{\Psi:[0,T]\times \Omega\to \L_2(U,V):\text{ predictable, }
 E\left[\int_0^T \norm{\Psi(s)}_{\L_2}^2\, ds\right]<\infty\right\},
\end{align*}
equipped with the same norm $\norm{\cdot}_{\H}$.

Let $(\Psi(t):\,t\in [0,T])$ be a simple process in $\H_0(\L_2)$ of the form \eqref{eq.defsimple} and $(L(t):\,t\ge 0)$ be
a cylindrical L{\'e}vy process in $U$ with weak second moments. According  to Theorem \ref{th.radonifying} we can
define $X_k:=J_{t_k,t_{k+1}}(\Phi_k)$ for each $k=0,\dots, N-1$ to obtain a random variable $X_k\in L^2_P(\Omega,\F;V)$
such that
\begin{align*}
  \big(L(t_{k+1})-L(t_k)\big)(\Phi_k^\ast v)
  =\scapro{X_k}{v}\qquad\text{for all }v\in V.
\end{align*}
Thus,  we can  define a random variable in $L^2_P(\Omega,\F;V)$ by
\begin{align*}
I(\Psi):\Omega\to V,
 \qquad I(\Psi):=  X_0 + \cdots + X_{N-1}
\end{align*}
and we call $I(\Psi)$ the {\em stochastic integral of $\Psi$ with respect to $L$}. Obviously, it satisfies
\begin{align*}
 \scapro{I(\Psi)}{v} = \sum_{k=0}^{N-1} \scapro{X_k}{v}
=  \sum_{k=0}^{N-1} \big(L(t_{k+1})-L(t_k)\big)(\Phi_k^\ast v)
\end{align*}
for all $v\in V$.

\begin{theorem}\label{th.intcont}
  The operator $I:\H_0(\L_2)\to L^2_P(\Omega,\F;V)$ is continuous.
\end{theorem}

Since the application of the cylindrical random variable $L(t)-L(s)$ to  a random variable in $U$
is defined by the extension of the operator $J_{s,t}$,
it is often convenient to work with a subspace of $\H_0(\L_2)$. For that purpose,  let $\H_{0}^{S}(\L_2)$
denote the space of all $\L_2(U,V)$-valued stochastic processes $(\Psi(t):\, t\in [0,T])$ which are of the form
\begin{align}
  \Psi(t)=\sum_{k=0}^{N-1} \Phi_k \1_{(t_k,t_{k+1}]}(t),
\end{align}
where $0=t_0< t_1< \cdots < t_N=T$ is a finite sequence of
deterministic times  and $\Phi_k$ is a  {\em simple} random variable  in $S(\Omega,\F_{t_k};\L_2)$ for each
$k=0,\dots, N-1$.
\begin{lemma}\label{le.simplesimple}
  For every $\Psi\in \H_0(\L_2)$ there exists a sequence $\{\Psi_n\}_{n\in\N}\subseteq \H_0^S(\L_2)$ such that  $\norm{\Psi-\Psi_n}_{\H}\to 0$ and $ \norm{I(\Psi)-I(\Psi_n)}_{L_P^2}\to 0$ as $n\to\infty$.
\end{lemma}
\begin{proof}
  Let $\Psi$ be a  simple stochastic process in $\H_0(\L_2)$ of the form
\begin{align*}
  \Psi(t)= \sum_{k=0}^{N-1} \Phi_k \1_{(t_k,t_{k+1}]}(t)\qquad\text{for all }t\in [0,T],
\end{align*}
where $0=t_0< t_1< \cdots < t_N=T$ is a finite sequence of deterministic times  and each
$\Phi_k$ is an arbitrary random  variable  in $L^2_P(\Omega,\F_{t_k};\L_2)$ for $k=0,\dots, N-1$.
 For each $k\in \{0,\dots, N-1\}$ there exists a sequence
$\{\Phi_{k,n}\}_{n\in\N}\subseteq S(\Omega,\F_{t_k};\L_2)$ of simple $\L_2(U,V)$-valued random variables  which converges to
$\Phi_k$ in $L^2_P(\Omega,\F_{t_k};\L_2)$. For each $n\in\N$ the stochastic process
$(\Psi_n(t):\, t\in [0,T])$ defined by
\begin{align*}
  \Psi_n(t):=\sum_{k=0}^{N-1} \Phi_{k,n}\1_{(t_k,t_{k+1}]}(t)\qquad\text{for all }t\in [0,T],
\end{align*}
is in $\H_0^S(\L_2)$ and the sequence $\{\Psi_n\}_{n\in\N}$ satisfies
\begin{align*}
\norm{\Psi-\Psi_n}_{\H}^2&=
  E\left[\int_0^T \norm{\Psi(s)-\Psi_n(s)}^2_{\L_2}\, ds\right]\\
  &= \sum_{k=0}^{N-1}(t_{k+1}-t_k) E\left[ \norm{\Phi_k-\Phi_{k,n}}^2_{\L_2}\right]
  \to 0 \qquad\text{as }n\to\infty .
\end{align*}
On the other  hand, the Cauchy-Schwarz inequality implies that
\begin{align*}
 \norm{I(\Psi)-I(\Psi_n)}^2_{L_P^2}
  &= \sum_{j=1}^\infty E\left[\abs{\sum_{k=0}^{N-1}\big(L(t_{k+1})-L(t_k)\big)(\Phi_k^\ast f_j- \Phi_{k,n}^\ast f_j)}^2\right]\\
&\le (N-1)\sum_{k=0}^{N-1} \sum_{j=1}^\infty
 E\left[\abs{\big(L(t_{k+1})-L(t_k)\big)(\Phi_k^\ast f_j- \Phi_{k,n}^\ast f_j)}^2\right]\\
&= (N-1)\sum_{k=0}^{N-1}
 E\left[\norm{J_{t_k,t_{k+1}}(\Phi_k - \Phi_{k,n} )}^2_V\right]\\
&\le  (N-1)\sum_{k=0}^{N-1}  \norm{J_{t_k,t_{k+1}}}^2_{L^2_P\to L^2_P} E\left[\norm{\Phi_k-\Phi_{k,n}}_{\L_2}^2\right]\\
&\to 0 \qquad\text{as }n\to\infty,
\end{align*}
where $J_{t_k,t_{k+1}}: L_P^2(\Omega,\F_{t_k};\L_2)\to L_P^2(\Omega,\F;V)$ is the operator defined in
\eqref{eq.defJ} for $\Delta L:=L(t_{k+1})-L(t_k)$.
\end{proof}

\begin{proof} (Theorem \ref{th.intcont})\\
Let $L$ has the cylindrical characteristics $(p,q,\nu)$.
Since the weak second moments exist, it follows from \cite[Cor. 3.12]{DaveMarkus} that
$L$ can be written as
\begin{align*}
  L(t)u
  =\tilde{p}(u)t + \tilde{L}(t)u\qquad\text{for all }t\ge 0,\, u\in U,
\end{align*}
where $\tilde{p}:U\to\R$ is a linear mapping and $(\tilde{L}(t):\,t\ge 0)$ is a cylindrical L{\'e}vy process, which is a cylindrical martingale with $\tilde{L}(0)=0$ $P$ a.s. Since $\tilde{p}(u)=E[L(1)u]$ for all $u\in U$, it follows from the continuity of $L(t):U\to L_P^2(\Omega;\F;\R)$ that $\tilde{p}$ is continuous, which in turn implies that $\tilde{L}(t):U\to L^2_P(\Omega,\F;\R)$ is continuous. The cylindrical process $\tilde{L}$ can be decomposed further into
$\tilde{L}(t)=W(t)+P(t)$. Here,  $(W(t):\, t\ge 0)$ is a Gaussian  cylindrical process\footnote{Note, that although $\tilde{L}(t):U\to L^2_P(\Omega,\F;\R)$ is continuous and $\tilde{L}=W+P$ the linear mappings $W(t):U\to L^2_P(\Omega,\F;\R)$ and $P(t):U\to L^2_P(\Omega,\F;\R)$ are not necessarily continuous. Nevertheless, we call $W$ and $P$ here cylindrical processes for simplicity although in all other situations we require continuity in the definition of a cylindrical random variable.} with characteristic function $\phi_{W(t)}(u) =\exp(-\tfrac{1}{2}t^2 q(u))$, $u\in U$ and $t\ge 0$, see \cite{Riedle:wiener}.
The cylindrical process $(P(t):\, t\ge 0)$ is independent of $W$ and is given by
\begin{align*}
  P(t)u=\int_{\R\setminus\{0\}} \beta\, \hat{N}_u(t, d\beta) \qquad
  \text{for all }u\in U,
  \end{align*}
where $\hat{N}_u$ denotes the compensated Poisson random measure with L{\'e}vy measure $\nu\circ \scapro{\cdot}{u}^{-1}$ of the Poisson random measure  defined by
\begin{align*}
  N_u(t,B):=\sum_{0\le s\le t} \1_B\big((L(s)-L(s-))u\big)
  \qquad\text{for all }B\in\Borel(\R\setminus\{0\}),\,t\ge 0.
\end{align*}
It follows for all $0\le s\le t$ and each $u\in U$ that
\begin{align*}
  &E\left[\abs{(W(t)-W(s))u}^2\right]=(t-s)q(u),\\
  &E\left[\abs{(P(t)-P(s))u}^2\right]=(t-s)\int_U\scapro{u}{a}^2\, \nu(da).
\end{align*}
Consequently,  for all $u\in U$ we obtain
\begin{align*}
  E\left[|\big(\tilde{L}(t)-\tilde{L}(s)\big)u|^2\right]&=
   (t-s)\Big( q(u)+ \int_U \scapro{u}{a}^2\,\nu(da)\Big)\\
   &=(t-s) E\left[|\tilde{L}(1)u|^2\right].
\end{align*}
Let $(\Psi(t):\, t\in [0,T])$ be a simple, $\L_2(U,V)$-valued stochastic process in $\H_0^S(\L_2)$ of the form
\begin{align}\label{eq.defpsi1}
  \Psi(t)= \sum_{k=0}^{N-1} \Phi_k \1_{(t_k,t_{k+1}]}(t),
\end{align}
where $0=t_0< t_1< \cdots < t_N=T$ is a finite sequence of deterministic times
 and each $\Phi_k$ is in $S(\Omega,\F_{t_k};\L_2)$ for $k=0,\dots, N-1$ and is of the form
\begin{align*}
\Phi_k=\sum_{i=1}^{m_k}  \1_{A_{i,k}} \phi_{i,k}
\end{align*}
for disjoint sets $A_{1,k},\dots, A_{m_k,k}\in\F_{t_k}$ and $\phi_{1,k},\dots, \phi_{m_k,k}\in
\L_2(U,V)$.
The martingale property of $\tilde{L}$ implies
the orthogonality of the increments; that is, for two times $0\le t_k\le t_\ell\le T$
and $\Delta \tilde{L}_i :=\tilde{L}(t_{i+1})-\tilde{L}(t_{i})$, $i=k,\ell$ it follows for all $u\in U$ that
\begin{align*}
& E\left[ (\Delta \tilde{L}_k)(\Phi_k^\ast u) (\Delta \tilde{L}_\ell)(\Phi_\ell^\ast u)\right]\\
&\qquad\qquad = E\left[ \left(\sum_{i=1}^{m_{k}}\1_{A_{i,k}}\Delta \tilde{L}_k (\phi_{i,k}^\ast u)\right)
    \left(\sum_{i=1}^{m_{\ell}}\1_{A_{i,\ell}}\Delta \tilde{L}_\ell (\phi_{i,\ell}^\ast u)\right)\right]\\
&\qquad\qquad = E\left[ \sum_{i=1}^{m_{k}}\sum_{i=1}^{m_{\ell}} \1_{A_{i,k}}\1_{A_{i,\ell}} \Delta \tilde{L}_k(\phi_{i,k}^\ast u)\,
    E\left[\Delta \tilde{L}_\ell (\phi_{i,\ell}^\ast u)|\F_{t_{\ell}}\right]\right]\\
&\qquad\qquad= 0.
\end{align*}
The decomposition  $L(t)u=\tilde{p}(u)t+ \tilde{L}(t)u$, $t\ge 0$, $u\in U$
enables us to estimate
\begin{align}\label{eq.estI_t(K)}
E\left[\norm{I(\Psi)}^2_V\right]
& =  \sum_{j=1}^\infty E\left[\abs{\scapro{I(\Psi)}{f_j}}^2\right]\notag\\
& = \sum_{j=1}^\infty E\left[\abs{\sum_{k=0}^{N-1} \Big(L(t_{k+1})-L(t_k)\Big)(\Phi^\ast_k f_j)}^2\right]\notag\\
& \le 2S(\Psi) +2 R(\Psi),
\end{align}
where the sums are defined by
\begin{align*}
S(\Psi)&:= \sum_{j=1}^\infty  E\left[\abs{\sum_{k=0}^{N-1}  \Big(\tilde{L}(t_{k+1})-\tilde{L}(t_k)\Big)(\Phi_k^\ast f_j)}^2\right],\\
R(\Psi)&:= \sum_{j=1}^\infty  E\left[\abs{\sum_{k=0}^{N-1}  \tilde{p}(\Phi^\ast_k f_j)(t_{k+1}-t_k)}^2\right].
 \end{align*}
By the orthogonal increments of $\tilde{L}$ and by using the equality \eqref{eq.condex}, which
  can be applied since $\Phi_k$ is a simple random variable in $S(\Omega,\F_{t_k};\L_2)$ for each $k=0,\dots, N-1$,
we can estimate the first sum in \eqref{eq.estI_t(K)} by
\begin{align*}
S(\Psi)
&=\sum_{j=1}^\infty \sum_{k=0}^{N-1}  E\left[\abs{ \Big(\tilde{L}(t_{k+1})-\tilde{L}(t_k)\Big)(\Phi_k^\ast f_j)}^2\right]\\
& = \sum_{j=1}^\infty \sum_{k=0}^{N-1}
\int_\Omega E\left[\abs{\Big(\tilde{L}(t_{k+1})-\tilde{L}(t_k)\Big)(\Phi_k^\ast (\omega) f_j)}^2\right]\,P(d\omega)\\
& \le  \sum_{j=1}^\infty \sum_{k=0}^{N-1} \norm{\tilde{L}(t_{k+1})-\tilde{L}(t_k)}_{U\to L^2_P}^2 E\left[ \norm{\Phi_k^\ast f_j}^2_U\right]\\
&= \norm{\tilde{L}(1)}_{U\to L^2_P}^2 \sum_{k=0}^{N-1} (t_{k+1}-t_k) E\left[ \norm{\Phi_k}_{\L_2}^2\right]\\
&= \norm{\tilde{L}(1)}_{U\to L^2_P}^2 E\left[ \int_0^T  \norm{\Psi(s)}_{\L_2}^2\,ds\right].
\end{align*}
The second sum in \eqref{eq.estI_t(K)} can be estimated by
\begin{align}\label{eq.estimateR}
R(\Psi)
&= \sum_{j=1}^\infty E\left[\abs{\int_0^T \tilde{p}(\Psi^\ast (s)f_j)\,ds}^2\right]\notag\\
&\le T\sum_{j=1}^\infty E\left[ \int_0^T \abs{\tilde{p}(\Psi^\ast (s)f_j)}^2\,ds \right]\notag\\
&\le T\norm{\tilde{p}}^2_U\sum_{j=1}^\infty E\left[ \int_0^T \norm{\Psi^\ast (s)f_j)}^2_U\,ds \right]\notag\\
&= T\norm{\tilde{p}}^2_U E\left[ \int_0^T \norm{\Psi (s)}^2_{\L_2}\,ds \right].
\end{align}
The last two estimates show the claim for each $\Psi\in \H_0^S(\L_2)$ of the form \eqref{eq.defpsi1}.

Let $\Psi $ be an arbitrary element in $\H_0(\L_2)$.
By Lemma \ref{le.simplesimple} there exists a sequence $\{\Psi_n\}_{n\in\N}\subseteq \H_0^S(\L_2)$
such that $\Psi_n\to \Psi$ in $\H(\L_2)$ and $I(\Psi_n)\to I(\Psi)$ in $L_P^2(\Omega,\F;V)$.  It follows from the first part of the proof that there exists a constant $c>0$, independent of $n$, such that
\begin{align}\label{eq.claimsimple}
  E\left[\norm{I(\Psi_n)}^2_V\right]
  \le c E\left[\int_0^T \norm{\Psi_n(s)}^2_{\L_2}\, ds\right]
  \qquad\text{for all }n\in\N.
\end{align}
The convergence of $\Psi_n$ and $\I(\Psi_n)$ completes the proof.
\end{proof}

Due to Theorem \ref{th.intcont} we are now in the classical situation and
the operator $I$ can be extended to a linear continuous operator $I:\H(\L_2)\to L_P^2(\Omega,\F;V)$ which enables us
to define:
\begin{align*}
  \int_0^T \Psi(s)\, dL(s):= I(\Psi) \qquad \text{for each }\Psi \in \H(\L_2).
\end{align*}

Since we follow It{\^o}'s approach in this work there is also an It{\^o} isometry in the following. In this result, we require  that the quadratic form $q:U\to \R$ of the cylindrical characteristics $(p,q,\nu)$ of the cylindrical L{\'e}vy process is continuous. This covers the standard situation, confer for example the case of a cylindrical Wiener process in \cite{Riedle:wiener}.

\begin{corollary}\label{co.Ito}
Let $L$ be a cylindrical L{\'e}vy process with cylindrical characteristics $(p,q,\nu)$ and with weak second moments.
If $L$ is a cylindrical martingale and the quadratic form $q:U\to \R$ is continuous  then for each $\Psi\in \H(\L_2)$ we have
\begin{align*}
    \norm{\int_0^T \Psi(s)\, dL(s)}_{L^2_P}^2= \int_0^T E\left[\norm{Q^{1/2}\Psi^\ast (s)}_{\L_2}^2
     + \int_{U}\norm{\Psi (s) u}_V^2\, \nu(du)\right]\,ds,
\end{align*}
where $Q\in \L(U,U)$ satisfies $q(u)=\scapro{Qu}{u}$ for each $u\in U$.
\end{corollary}

\begin{proof}
Since $L$ is a cylindrical martingale it can be decomposed
into $L(t)u= W(t)u+ P(t)u$ for all $t\ge 0$ and $u\in U$, where $W$ and $P$ are defined in the proof of Theorem \ref{th.intcont}. However, since $q$ is required to be continuous the mapping $W(t):U\to L^2_P(\Omega;\R)$ is continuous and thus, also $P(t):U\to L^2_P(\Omega;\R)$.
It follows for all $0\le s\le t$ and each $u\in U$ that
\begin{align*}
  &E\left[\abs{(W(t)-W(s))u}^2\right]=(t-s)\scapro{Qu}{u},\\
  &E\left[\abs{(P(t)-P(s))u}^2\right]=(t-s)\int_U\scapro{u}{a}^2\, \nu(da).
\end{align*}
Let $\Psi\in \H_0^S(\L_2)$ be of the form \eqref{eq.defpsi1}.
The orthogonal increments of $(L(t):\, t\ge 0)$
imply as in the proof of Theorem \ref{th.intcont} that
\begin{align}\label{eq.Itocal}
&E\left[\norm{\int_0^T \Psi(s)\, dL(s)}_V^2\right]\notag\\
&\qquad =\sum_{j=1}^\infty \sum_{k=0}^{N-1}  E\left[\abs{ \Big(L(t_{k+1})-L(t_k)\Big)(\Phi_k^\ast f_j)}^2\right]\notag\\
&\qquad =\sum_{j=1}^\infty\sum_{k=0}^{N-1} \sum_{i=1}^{m_k}
 \Big(E\left[ \1_{A_{i,k}}\abs{(W(t_{k+1})-W(t_k))(\phi^\ast_{i,k} f_j)\right.\right.\notag\\
&\hspace*{6cm}\left.\left.\left.    + (P(t_{k+1})-P(t_k))(\phi^\ast_{i,k} f_j)
}^2\right]\right) \notag\\
&\qquad =\sum_{j=1}^\infty\sum_{k=0}^{N-1}(t_{k+1}-t_k) \sum_{i=1}^{m_k}
E[\1_{A_{i,k}}]  \left( \scapro{Q\phi_{i,k}^\ast f_j}{\phi_{i,k}^\ast f_j}
  + \int_U\scapro{\phi_{i,k}^\ast f_j}{u}^2\, \nu(du)\right) \notag\\
&\qquad =\sum_{j=1}^\infty\sum_{k=0}^{N-1}(t_{k+1}-t_k)
E\left[ \scapro{Q^{1/2}\Phi_{k}^\ast f_j}{Q^{1/2}\Phi_{k}^\ast f_j}
  + \int_U\scapro{\Phi_{k}^\ast f_j}{u}^2\, \nu(du)\right] \notag\\
&\qquad =\sum_{k=0}^{N-1}(t_{k+1}-t_k)
   E\left[\norm{Q^{1/2}\Phi_k^\ast}_{\L_2}^2
  + \int_U \norm{\Phi_k u}_V^2  \, \nu(du)\right] \notag\\
&\qquad =\int_0^T\left(E\left[\norm{Q^{1/2}\Psi^\ast (s)}_{\L_2}^2\right]
     + \int_{U}\norm{\Psi (s) u}_V^2\, \nu(du)\right)\,ds,
\end{align}
where we used the fact that $Q$ can be factorised into
$Q=(Q^{1/2})^\ast Q^{1/2}$.

 Since the cylindrical measure $\nu$ is per se only defined on  the algebra $\Z(U)$ whereas the norm is not measurable with respect to $\Z(U)$ the integral in \eqref{eq.Itocal} requires a clarification:
for a Hilbert-Schmidt operator $\phi\in \L_2(U,V)$ the image cylindrical measure $\nu\circ \phi^{-1}$ extends to a Radon measure
with finite second moment. Thus, we can define
\begin{align*}
  \int_U \norm{\phi u}_V^2 \, \nu(du):=\int_V \norm{v}_V^2\, (\nu\circ \phi^{-1})(dv)<\infty.
\end{align*}
For fixed $\omega\in \Omega$,  this argument can be extended by linearity to the integration of $\norm{\Phi_k (\omega) u}_V^2$ for $\Phi_k\in S(\Omega,\F_{t_k};\L_2)$  and
$\norm{\Psi(s)(\omega)u}_V^2$ for $\Psi\in \H_0^S(\L_2)$ and $u\in U$ with respect to the cylindrical measure $\nu$ in \eqref{eq.Itocal}.

Let a sequence $\{\Psi_n\}_{n\in\N}\subseteq \H_0^S(\L_2)$ converges to $\Psi\in \H(\L_2)$.
Since $(P(t):\, t\ge 0)$ is also a cylindrical L{\'e}vy process, we obtain by the same calculation
as in  \eqref{eq.Itocal}  that
 \begin{align*}
E\left[\norm{\int_0^T \Psi_n(s)\, dP(s)}^2_V\right]=E\left[ \int_0^T \int_U\norm{\Psi_n(s)u}_V^2\, \nu(du)ds\right].
 \end{align*}
 Consequently, we can define
\begin{align*}
  \int_0^T \int_U \norm{\Psi(s)u}_V^2 \, \nu(du)ds
  :=\lim_{n\to\infty}  \int_0^T \int_U\norm{\Psi_n(s)u}_V^2\, \nu(du)ds,
\end{align*}
where the convergence takes place in $L_P^1(\Omega,\F;\R)$
Since the convergence of all other terms in the first and last line in
\eqref{eq.Itocal} is obvious, an application of Lemma \ref{le.simplesimple}
completes the proof.
\end{proof}

In order to obtain a stochastic process we define for $\Psi\in \H(\L_2)$
\begin{align*}
  \int_0^t \Psi(s)\, dL(s):=\int_0^T \1_{[0,t]}(s)\Psi(s)\, dL(s)
  \qquad  \text{for  }t\in [0,T].
\end{align*}

The following result completes the introduction of the integral with respect to cylindrical L{\'e}vy processes.
\begin{corollary}\label{co.process}
If $(L(t):\,t\ge 0)$ is a cylindrical L{\'e}vy process with weak second moments then for every $\Psi\in\H(\L_2)$ there exists a $V$-valued adapted stochastic process $I$ with strong second moments and with c\`adl\`ag paths in $V$ which is a modification of
\begin{align*}
  \left(\int_0^t \Psi(s)\, dL(s):\, t\in [0,T]\right).
\end{align*}
If in addition, $(L(t):\, t\ge 0)$ is a cylindrical martingale then
$I$ is a martingale in $V$ with strong second moments.
\end{corollary}
\begin{proof}
  Decompose $L$ as in the proof of Theorem \ref{th.intcont} into the cylindrical martingale $\tilde{L}$
  and the linear, continuous mapping $\tilde{p}:U\to \R$, i.e. $L(t)u=\tilde{p}(u)t+\tilde{L}(t)u$, $t\ge 0$ for all $u\in U$.
Since $\tilde{L}$ is again a cylindrical L{\'e}vy process the
stochastic integral $\tilde{I}(\Psi):=\int_0^T \Psi(s)\, d\tilde{L}(s)$ is well defined for $\Psi \in \H(\L_2)$
and we obtain the decomposition
  \begin{align}\label{eq.intdecomp}
I(\Psi)= \tilde{I}(\Psi) + F(\Psi),
  \end{align}
 where $F(\Psi)$ is defined as the $V$-valued random variable $F(\Psi):=I(\Psi) - \tilde{I}(\Psi)$.

 For $\Psi\in \H_0^S(\L_2)$
 the random variable $F$ satisfies
\begin{align}\label{eq.F}
  \scapro{F(\Psi)}{v}=\int_0^T \tilde{p}(\Psi^\ast (s)v)\, ds
   \qquad\text{for all }v\in V.
\end{align}
If a sequence $\{\Psi_n\}_{n\in\N}\subseteq \H_0^S(\L_2)$ converges to $\Psi$ in $\H(\L_2)$ then $F(\Psi_n)\to F(\Psi)$ in $L^2_P(\Omega,\F;V)$ by the continuity of $I$ and $\tilde{I}$.
On the other hand, by changing to a subsequence if necessary we can assume that
\begin{align*}
  \int_0^T \norm{\Psi_n(s)-\Psi(s)}_{\L_2}^2\,ds \to 0
  \quad\text{$P$-a.s. as $n\to\infty$.}
\end{align*}
Together with a similar estimate as in \eqref{eq.estimateR}, this implies for each $v\in V$ that
\begin{align*}
\abs{\int_0^T \tilde{p}\big((\Psi_n^\ast(s)-\Psi^\ast(s))v\big)\,ds }^2
&\le T\norm{\tilde{p}}^2_U \int_0^T \norm{(\Psi_n^\ast- \Psi^\ast) (s)v}^2_U\,ds \notag\\
&\le T\norm{\tilde{p}}^2_U \norm{v}^2_V  \int_0^T \norm{(\Psi_n- \Psi)(s)}^2_{\L_2}\,ds \\
&\to 0 \quad\text{$P$-a.s. as $n\to\infty$},
\end{align*}
which shows that \eqref{eq.F} is also satisfied for all  $\Psi\in \H(\L_2)$.

Define $\tilde{I}(\Psi)(t):=\tilde{I}(\Psi \1_{[0,t]})$ for $\Psi\in \H(\L_2)$.
If $\Psi\in \H_0^S(\L_2)$ it
is easy to see that the stochastic process  $\tilde{M}(\Psi):=(\tilde{I}(\Psi)(t):\, t\in [0,T])$ is an $V$-valued martingale. Due to the continuity of
the conditional expectation as an operator on $L_P^2(\Omega,\F;V)$ one can extend this result to
$\Psi\in \H(\L_2)$ by  Theorem \ref{th.intcont} and Lemma \ref{le.simplesimple}.

For every $\Psi\in \H(\L_2)$ we have
\begin{align*}
  \norm{\Psi \1_{(s,t]}}_{\H}^2
  =\int_s^t E[\norm{\Psi(r)}_{\L_2}^2]\, dr
  \to 0 \qquad\text{as }s\to t.
\end{align*}
Therefore, the continuity of the integral operator $\tilde{I}:\H(\L_2)\to L_P^2(\Omega,\F;V)$ implies for each $\epsilon>0$ that
\begin{align*}
  P\left(\norm{\tilde{\I}(\Psi)(t)-\tilde{I}(\Psi)(s)}_V\ge \epsilon\right)
&\le \frac{1}{\epsilon^2}
 E\left[ \norm{\tilde{\I}(\Psi)(t)-\tilde{I}(\Psi)(s)}_V^2\right] \\
& =\frac{1}{\epsilon^2}
  \norm{\tilde{I}(\Psi \1_{(s,t]})}_{L^2}^2\\
&\to 0 \qquad \text{as }s\to t,
\end{align*}
which shows that the $V$-valued martingale
$\tilde{M}(\Psi)$ is continuous in probability. It follows by an infinite dimensional version
of Doob's regularisation theorem (see e.g. \cite[Th. 3.13]{PeszatZab}) that there exists a modification of $\tilde{M}(\Psi)$ with c\`adl\`ag paths.

Since $\tilde{p}$ is linear it follows from \eqref{eq.F} for $\Psi\in \H(\L_2)$ and $0\le s\le t$ that
\begin{align*}
  \norm{F(\Psi\1_{[0,t]})-F(\Psi\1_{[0,s]})}_V
  &= \sum_{j=1}^\infty \abs{\int_s^t \tilde{p}(\Psi^\ast (r)f_j )\, dr}^2\\
  &\le (t-s) \sum_{j=1}^\infty \int_s^t \abs{\tilde{p}(\Psi^\ast(r)f_j)}^2\, dr \\
  &\le (t-s)\norm{\tilde{p}}_{U}^2 \int_0^T \norm{\Psi(r)}_{\L_2}^2\, dr.
\end{align*}
Consequently, the stochastic process $(F(\Psi \1_{[0,t]}): \, t\in [0,T])$ has continuous trajectories.
 Due to the  decomposition \eqref{eq.intdecomp} this shows the existence of a modification of $(\int_0^t \Psi(s)\, dL(s):\, t\in [0,T])$  with c\`adl\`ag paths for each $\Psi\in \H(\L_2)$.

If $L$ is a cylindrical martingale it follows $\tilde{p}=0$ in the above calculations, which shows the second claim.
\end{proof}

\section{Differential Equations}\label{se.diff}

The developed theory of stochastic integration can be applied to
a stochastic partial differential equation  of the form
\begin{align}\label{eq.sde}
\begin{split}
  dX(t)&= \left( AX(t)+ F(X(t))\right)\, dt + G(X(t))\, dL(t) \qquad\text{for all }t\in [0,T],\\
   X(0)&=X_0,
  \end{split}
   \end{align}
where $A:\text{dom}(A)\subseteq V\to V$ is the generator of a $C_0$-semigroup $(S(t))_{t\ge 0}$ on $V$,
$F:V\to V$ and $G:V\to \L(U,V)$ are some mappings and the initial condition $X_0$ is
an $\F_0$-measurable random variable in $L^2_P(\Omega,\F_0;V)$. The random noise $(L(t):\,t\in [0,T])$ is a cylindrical L{\'e}vy process in $U$  with weak second moments.

Since the stochastic integral with respect to the generalised noise $(L(t):\, t\in [0,T])$ is developed in the last sections one can
follow very closely the existing literature to derive existence, uniqueness and further properties of solutions of \eqref{eq.sde} or even more general equations. For that reason, we keep this section very short and indicate only the approach in the standard situation.

A predictable $V$-valued stochastic process $(X(t):\, t\in [0,T])$ is
called a {\em mild solution of Equation \eqref{eq.sde}} if for all $t\in [0,T]$ it satisfies $P$-a.s. that
\begin{align*}
   X(t)=S(t)X_0+\int_0^t S(t-s)F(X(s))\, ds + \int_0^t S(t-s)G(X(s))\, dL(s).
\end{align*}
The last integral on the right hand side is the stochastic integral developed in the previous sections.

In the following result we employ the standard assumptions, linear growth and Lipschitz conditions, which guarantee the existence of a mild solution, see for example
\cite[p. 65]{DaPratoZab2} or \cite[p. 148]{PeszatZab}.
\begin{theorem}
If there exists a constant $c>0$ such that for all $v_1,v_2\in V$ and
all $t\in [0,T]$ the operators satisfy
\begin{align*}
  \norm{S(t)(Fv_1)}_V + \norm{S(t)(Gv_1)}_{\L_2}\le c(1+\norm{v_1}_V),\\
\norm{S(t)(F(v_1)-F(v_2))}_V+ \norm{S(t)(G(v_1)-G(v_2))}_{\L_2}
\le c\norm{v_1-v_2}_V,
\end{align*}
then there exists a unique (up to modification) mild solution $X$
of \eqref{eq.sde} satisfying $\sup_{t\in [0,T]}E\left[\norm{X(t)}_V^2\right]<\infty$.
\end{theorem}
\begin{proof}
The standard proof  applies Banach's fixed point theorem
in the space
\begin{align*}
  \mathcal{X_{T,\beta}}
  :=\left\{Y:[0,T]\times\Omega\to V\text{ predictable and}
    \sup_{t\in [0,T]} E\left[\norm{Y(t)}_V^2\right]<\infty\right\},
\end{align*}
and  equipped for some $\beta>0$ with the norm
\begin{align*}
  \norm{Y}_{T,\beta}:= \left(\sup_{t\in [0,T]} e^{-\beta t} E\left[\norm{Y(t)}_V^2\right]\right)^{1/2}.
\end{align*}
The operator, which will be shown to be a contraction for an appropriate $\beta>0$, is defined by
\begin{align*}
  K: {\mathcal X}_{T,\beta}\to {\mathcal X}_{T,\beta}, \qquad K(Y):=K_1(Y)+K_2(Y),
\end{align*}
where for all $t\in [0,T]$ we define
\begin{align*}
 & K_1(Y)(t):=\int_0^t S(t-s)F(Y(s))\, ds\\
 & K_2(Y)(t):=\int_0^t S(t-s)G(Y(s))\,dL(s).
\end{align*}
The integrals, and in particular the newly developed stochastic integral, are well defined due to the assumed linear growth condition. There
exists a predictable version of the stochastic
processes $(K_1(Y)(t):\, t\in [0,T])$ and $(K_2(Y)(t):\, t\in [0,T])$, since they are adapted and stochastically continuous. For $K_2(Y)$ this can be seen as in the proof of Corollary \ref{co.process}.

The Lipschitz continuity shows that there exists a constant $c_1>0$ such that
\begin{align}\label{eq.contraction1}
 \norm{K_1(Y_1)-K_1(Y_2)}_{T,\beta}^2\le c_1 \norm{Y_1-Y_2}_{T,\beta}^2
 \qquad\text{for all }Y_1,Y_2\in {\mathcal X}_{T,\beta}.
\end{align}
The continuity of the integral operator $I:\H(\L_2)\to L_P^2 (\Omega,\F;V)$ implies that there exists a constant $c_2$ such that
\begin{align*}
 &\norm{K_2(Y_1)-K_2(Y_2)}_{T,\beta}^2\\
 &\qquad\qquad = \sup_{t\in [0,T]}e^{-\beta t} \norm{\int_0^t S(t-s)(G(Y_1(s))-G(Y_2(s)))\,dL(s)}_{L^2_P}^2\\
 &\qquad\qquad \le c_2 \sup_{t\in [0,T]}e^{-\beta t} E\left[\int_0^t \norm{ S(t-s)(G(Y_1(s))-G(Y_2(s)))}_{\L_2}^2\,ds\right].
\end{align*}
Applying the Lipschitz continuity results in
\begin{align}\label{eq.contraction2}
 \norm{K_2(Y_1)-K_2(Y_2)}_{T,\beta}^2\le c_3 \norm{Y_1-Y_2}_{T,\beta}^2
 \qquad\text{for all }Y_1,Y_2\in {\mathcal X}_{T,\beta},
\end{align}
for a constant $c_3>0$. Both inequalities \eqref{eq.contraction1} and
\eqref{eq.contraction2} show that $K$ is a contraction if we choose $\beta$ large enough,  which completes the proof.
\end{proof}

{\bf Acknowledgement:} This work was initiated by an inspiring discussion with A. Jakubowski in Ascona 2011 for
which I thank him.

\bibliographystyle{plain}

\end{document}